\documentclass[11pt,a4paper,reqno]{amsart}

\usepackage{color,latexsym,amsfonts,amssymb,bbm,comment,amsmath,cite, amsthm}
\usepackage{hyperref}
\usepackage{fullpage}
\usepackage{fancyhdr}
\usepackage{graphicx}
\usepackage{tikz}
\usepackage{todonotes}
\usepackage{dsfont}

\providecommand{\R}{\mathbb{R}}
\providecommand{\Rd}{{\mathbb{R}^d}}
\newcommand{\Zd}{{\mathbb{Z}^d}}

\renewcommand{\P}{\mathbb{P}}
\providecommand{\oP}{\overline{\mathbb{P}}}
\providecommand{\oE}{\overline{\mathbb{E}}}

\providecommand{\one}{\mathbbmss{1}}
\providecommand{\eps}{\varepsilon}

\renewcommand{\t}{t}
\providecommand{\s}{\sigma}

\newcommand{\blue}[1]{{\textcolor{blue}{{\bf \emph{#1}}}}{}}

\newtheorem{theorem}{Theorem}[section]

\newtheorem{lemma}[theorem]{Lemma}
\newtheorem{proposition}[theorem]{Proposition}
\theoremstyle{definition}


\begin{document}

\author{Markus Heydenreich}
\author{Christian Hirsch}
\address[Markus Heydenreich, Christian Hirsch]{Mathematisches Institut, Ludwig-Maximilians-Universit\"at M\"unchen, Theresienstra\ss e 39, 80333 M\"unchen, Germany}
\email{m.heydenreich@lmu.de,christian.hirsch@lmu.de}
\author{Daniel Valesin}
\address[Daniel Valesin]{Johann Bernoulli Institute, University of Groningen, Nijenborgh 9, 9747AG Groningen, The Netherlands}
\email{D.Rodrigues.Valesin@rug.nl}

\thanks{The work of CH was funded by LMU Munich's Institutional Strategy LMUexcellent within the framework of the German Excellence Initiative.}

\title{Uniformity of hitting times of the contact process}

\date{\today}

\begin{abstract}
For the supercritical contact process on the hyper-cubic lattice started from a single infection at the origin and conditioned on survival, we establish two uniformity results for the hitting times $t(x)$, defined for each site $x$ as the first time at which it becomes infected. First, the family of random variables $(t(x)-t(y))/|x-y|$, indexed by $x \neq y$ in $\mathbb{Z}^d$, is stochastically tight. Second, for each $\varepsilon >0$ there exists $x$ such that, for infinitely many integers $n$, $t(nx) < t((n+1)x)$ with probability larger than $1-\varepsilon$.  A key ingredient in our proofs is a tightness result concerning the essential hitting times of the supercritical contact process introduced by Garet and Marchand (Ann.\ Appl.\ Probab., 2012).\end{abstract}

\maketitle


\section{Introduction}
\subsection{The contact process}
The contact process $\{\xi_t\}_{t \ge 0}$ is a Feller process on the configuration space $\{0,1\}^\Zd$ with generator
\begin{equation}
	\mathcal Lf(\eta)=\sum_{x\in\Zd}c(x,\eta)\big[f(\eta_x)-f(\eta)\big],
	\qquad \eta\in\{0,1\}^\Zd, 
\end{equation}
where $f\colon\{0, 1\}^\Zd \to \R$ is a cylinder function and $\eta_x$ is the configuration $\eta$ ``flipped at $x$'', i.e., 
\[ \eta_x(y)=\begin{cases}\eta(y)\quad&\text{if }y\neq x,\\ 1-\eta(y)\quad&\text{if }y=x.\end{cases}\] 
The ``flip rate'' $c(x,\eta)$ of the contact process is defined by 
\begin{equation}\label{eqFlipRateDef}
	c(x,\eta)=
	\begin{cases} 
		1&\text{if }\eta(x) = 1;\\
		\lambda \#\{y\in\Zd\colon|x - y| = 1, \eta(y) = 1\}&\text{if }\eta(x) = 0,
	\end{cases}
\end{equation}
where $|\cdot|$ denotes the 1-norm on $\Zd$, $\#S$ denotes the cardinality of a set $S$ and $\lambda \ge 0$ is a parameter of the process. 
The corresponding probability measure is denoted $\P_\lambda$. 
We refer to Liggett \cite[Chapter VI]{Ligge85} or \cite[Part I]{Ligge99} for a formal description, and for further references to Durrett \cite{Durre91}. 

The contact process is one of the prime examples of an attractive interacting particle system. Considering this process as a model for the spread of an infection leads to an insightful interpretation: lattice sites $x\in\Zd$ represent individuals, and at any given time $t\ge0$ individuals are either \emph{healthy} (if in state 0) or \emph{infected} (if in state 1). Infected individuals become healthy at rate 1, but spread the infection to their neighbors at rate $\lambda$. The higher $\lambda$, the more contagious the infection. 

The occurrence of a phase transition in $\lambda$ has been one of the driving forces for intense research activities centered around the contact process.
In order to describe this phenomenon, let $o$ denote the origin of $\mathbb{Z}^d$ and $\{\xi^o_t\}_{t\ge0}$ denote the contact process started from the configuration with a single infected individual at $o$, that is $\xi_0=\mathds{1}_{\{o\}}$, where $\mathds{1}$ is the indicator function. Then write 
\[ \{o \rightsquigarrow \infty\} = \bigcap_{t\ge0}\big\{\exists y\in\Zd\colon \xi^o_t(y) = 1\big\} \]
for the event that the infection \emph{survives} (in the complementary event, the infection is said to \textit{die out}). 
Then, there exists $\lambda_c\in(0,\infty)$ such that $\{\xi^o_t\}_{t \ge 0}$ has positive probability of survival if $\lambda > \lambda_c$ and dies out almost surely if $\lambda \le \lambda_c$. In other words, letting 
$$\rho = \P_\lambda(o \rightsquigarrow \infty),$$
denote the survival probability, we have 		$\rho	=0\text{ if }\lambda\le\lambda_c$ and $		\rho	>0\text{ if }\lambda>\lambda_c$, cf.\ \cite{Durre91,BezuiGrimm90}. 

\subsection{Supercritical contact process conditioned on survival}
As is standard, we always assume that the contact process is defined through a graphical representation (a family of Poisson processes consisting of \textit{transmission arrows} and \textit{recovery marks}, which induce \textit{infection paths}). We refrain from giving the details of the construction; rather, we briefly describe the terminology and notation we use. Given $x, y \in \Zd$ and $0 < s < t$, we write $(x,s) \rightsquigarrow (y,t)$ if there is an infection path from $(x,s)$ to $(y,t)$. In case $s = 0$, we simply write $x \rightsquigarrow (y,t)$. If $B, B' \subset \Zd \times [0,\infty)$, we write $B \rightsquigarrow B'$ if there is an infection path from a point in $B$ to a point in $B'$. The expressions $\{(x,s) \rightsquigarrow \infty\}$ and $\{x \rightsquigarrow \infty\}$ represent the events that there exists an infinite infection path started from $(x,s)$ and $(x,0)$, respectively.

 The graphical representation allows for the definition of the contact process with all possible initial configurations in the same probability space; one simply sets
$$\xi^A_t(x) =  \mathds{1}\{A \times \{0\} \rightsquigarrow (x,t) \},\quad A \subset \Zd,\;x \in \Zd,\; t \geq 0$$
for the process with $\xi^A_0 = \mathds{1}_A$. 

We often abuse notation and associate a configuration $\eta \in \{0,1\}^\Zd$ with the set $\{x:\eta(x) = 1\}$.

Throughout the paper, we work in the \emph{supercritical} regime. That is, we fix $\lambda>\lambda_c$ and write $\P$ instead of $\P_{\lambda}$ in the following. For these values of $\lambda$, write 
\begin{equation}
	\oP(E) = \P(E\mid o \rightsquigarrow \infty)
\end{equation}
for any measurable event $E$ on graphical constructions. That is, $\oP$ denotes the conditional probability given survival of the infection starting from the origin at time 0. We denote by $\oE$ the corresponding expectation. 
Mind that the Markov property, valid under the measure $\P$, fails under the conditional measure $\oP$. For this reason, in the proofs of our main results, we often need to switch back and forth between these two measures.

\subsection{Results}
Let 
$$t(x) = \inf\{s \ge 0:\, \xi^o_s = x\}$$ 
denote the first time that the contact process starting with a single infection at the origin $o$ infects the site $x$. We call this the \emph{hitting time} of $x$. 
In particular, the random field $\{t(x)\}_{x\in\Zd}$ contains fine-grained information on the space-time evolution of the contact process. The most prominent result is the \emph{shape theorem} for hitting times \cite{BezuiGrimm90, Durre88, DurreGriff82} describing asymptotically the ``once infected area'' 
\[H_t=\{x\in\Zd\colon t(x)\le t\}.\] 
Indeed, the normalized expected asymptotic hitting times
$$\mu(x)= \lim_{n \to \infty} \frac{t(nx)}{n}$$
define a norm on $\Rd$ satisfying that for all $\eps > 0$ $\oP$-a.s.\ for all $t\ge0$ large enough, \begin{equation}
	\{x\in\Zd\colon \mu(x)\le (1-\eps)t\}
	\subset H_t
	\subset \{x\in\Zd\colon \mu(x)\le (1+\eps)t\}.
\end{equation}
Additionally,  the probability that this event fails to hold for a given large time exhibits large deviation behavior~\cite{GaretMarch14} and $\mu(x)$ is continuous in the infection rate $\lambda$ for every $x\in\Rd$~\cite{GaretMarchThere15}. 

The shape theorem and its variants describe the global picture of the infection, that is, how fast the infection is spreading through space macroscopically. What can be said about the evolution on a microscopic level is not covered. A priori, it is possible that certain branches of the infection grow very fast at times so that the set $H_t$ exhibits holes close to its boundary. Moreover, the shape theorem alone does not make predictions on how wildly hitting times fluctuate locally or to what degree of monotonicity sites are visited in the order of increasing distance to the origin. In this paper, we refine this analysis by zooming in on the boundary of the set $H_t$ and prove that $H_t$ grows in a highly uniform way. 

First, we consider the order in which sites on a linear ray $\{nx\}_{n \ge 0}$ are hit by the contact process. We show that, after averaging over $n$, the probability that sites are visited in the correct order becomes arbitrarily close to 1 when choosing the step size $|x|$ sufficiently large. Moreover, also the points of increase of the expected hitting time function $\oE[t(kx)]$ are of density arbitrarily close to 1 when $|x|$ is sufficiently large.
\begin{theorem}[Uniformity of hitting times]
	\label{mainThm}
	For the supercritical contact process conditioned on survival, 
	$$\lim_{|x| \to \infty}\liminf_{n \to \infty} \frac1n \sum_{k=1}^n \oP\big(t({(k-1)x}) \le t({k x})\big) = 1,$$
	and 
	$$\lim_{|x| \to \infty}\liminf_{n \to \infty} \frac1n \#\{k \in \{1,\ldots ,n\}:\, \oE[t({(k-1)x})] \le \oE[t({kx})] \} = 1.$$
\end{theorem}

The Ces\`aro limit of Theorem \ref{mainThm} implies immediately that for every $\eps > 0$,
\begin{equation}\label{eq:limsup}
	\limsup_{n \to \infty} \oP\big(t({nx}) \le t({(n+1) x})\big)>1-\eps 
\end{equation}
provided $|x|$ is large enough. 

Our second result is a tightness result which shows that also locally, the fluctuations of $t(x) - t(y)$ are of linear order in $|x - y|$.  
\begin{theorem}[Tightness]\label{thmTightness}
For any $p>0$, 
	$$\sup_{\substack{x,y\in\Zd \\ x \ne y}} \frac{\oE\left[|t(x) - t(y)|^p \right]}{|x-y|^p}  < \infty.$$
	In particular, the family
	$$\left\{\frac{t(x) - t(y)}{|x-y|}\right\}_{x, y \in \Zd, x \ne y}$$
is {tight} under the measure $\oP$. 
\end{theorem} 

\subsection{Outline of proof and essential hitting times}
\label{essHitSec}
On a high level, both Theorems~\ref{mainThm} and~\ref{thmTightness} deal with bounds on differences of the form $t(x) - t(y)$. The first reflex in this situation would be to use the inequalities
\begin{align*}&(t(x) - t(y)) \cdot \mathds{1}\{t(y)<t(x)\}  \leq \tilde t(y,x) \cdot \mathds{1}\{t(y)<t(x)\},\\ &(t(y) - t(x)) \cdot \mathds{1}\{t(x)<t(y)\}  \leq \tilde t(x,y) \cdot \mathds{1}\{t(x)<t(y)\},\end{align*}
where $\tilde t(y,x)$ is the amount of time it takes for the infection spreading from $(y,t(y))$ to reach $x$, and similarly for $\tilde t(x,y)$. One is tempted to combine these inequalities with a translation invariance property: under $\P$, conditioning on say $t(y) < t(x)$, and after a space-time shift, $\tilde t(y,x)$ is distributed as $t(x-y)$.

Although true, these observations are not as useful as it would seem. There are two problems. First, the time $\tilde t(y,x)$ could be infinite, as the infection first reaching $y$ could die out before reaching $x$; this renders the above inequalities less useful. Second, under the measure $\oP$, the aforementioned translation invariance property is lost.


These problems illustrate why the shape theorem is substantially more difficult to establish for the contact process than for the related model of first-passage percolation (more about it below). It was only with the 
introduction of 
\emph{essential hitting times}~\cite{GaretMarch12} that the supercritical contact process could be analyzed via
sub-additivity techniques. 
The essential hitting times exhibit highly useful invariance properties under the conditional law $\oP$, and are therefore a central quantity in our analysis.

Loosely speaking, the essential hitting time is the first time when the infection hits a site $x$ and this particular branch of the infection survives. These are obviously not stopping times (because we are conditioning on the future), but it is possible to set up the definitions in a neat way so that the times $\sigma(x)$ are renewal times. Birkner et al.\ \cite{BirknCernyDeppeGante13} applied a similar construction in the context of a genealogical model. 

Now, we represent the expression $t(x) - t(y)$ as
$$t(x) - t(y) = (t(x) - \sigma(y)) + (\sigma(y) - t(y)).$$
As we will see in Section~\ref{thmSec}, in the specific setting of Theorems~\ref{mainThm} and~\ref{thmTightness}, the first summand is indeed amenable to a renewal argument. However, this approach is bound to fail for the second summand, since $t(y)$ is at most as big as $\sigma(y)$. The proof of the shape theorem relies crucially on bounds for $\sigma(y) - t(y)$~\cite[Proposition 17]{GaretMarch12}. However, those bounds turn out to be unsuitable for the granularity that we are aiming for, since they depend on the location $y$. To remove this dependence, we present a conceptually novel argument forming the heart of the proof for our main results.

In order to define the essential hitting times precisely, we define for any site $x\in\Zd$ a sequence of stopping times $\{u_n(x)\}_{n \ge 0}$ and $\{v_n(x)\}_{n \ge 0}$. 
These are initialised as $u_0(x)=v_0(x)=0$. 
Given $v_k(x)$, we define 
\[u_{k+1}(x)=\inf\{t\ge v_k(x):x\in\xi^o_t\}\]
to be the first time that $x$ is infected after time $v_k(x)$ (with the usual convention $\inf\varnothing=+\infty$). 
Further, given $u_{k+1}(x)$, we define 
\[v_{k+1}(x)=\sup\{s\ge0:\; (x,u_{k+1}(x)) \rightsquigarrow \Zd \times \{s\}\} \]
to be the time at which the single infection present at $x$ at time $u_{k+1}(x)$ dies out again. 

With these definitions at hand, we let 
\[K(x)=\min\{n\ge0:v_n(x)=\infty\text{ or }u_n(x)=\infty\}.\]
Note that, under $\oP$, we almost surely have $K(x) < \infty$, $u_{K(x)} < \infty$ and $v_{K(x)} = \infty$. Then let 
\[\sigma(x)=u_{K(x)}(x).\]

Clearly, the essential hitting times are larger than the standard ones, i.e., $\sigma(x) - t(x) \ge 0$. Additionally, for a given site $x \in \Zd$ the deviation of $\sigma(x)$ from $t(x)$ has exponential tails~\cite[Proposition 17]{GaretMarch12}, where the rate of decay could still depend on $x$. As a key tool in the proof of our main results, we show that in situations where stretched exponential decay is sufficient, this dependence can be entirely eliminated. 

\begin{proposition}\label{propEssTightness}
	There exist $C,\gamma > 0$ such that, for all $x \in \mathbb{Z}^d$ and all $L > 0$,
$$\oP(\sigma(x) - t(x) > L) < C\exp\left\{-L^\gamma\right\}.$$
\end{proposition}
An immediate consequence is that the differences between hitting times and essential hitting times $\{\sigma(x)-t(x)\}_{x\in\Zd}$ is a tight family of random variables. 

\subsection{Discussion and open problems}
\subsubsection{First-passage percolation and coexistence for competition models} 


In first-passage percolation, every bond $\{x,y\}$ (with $|x-y|=1$) in $\Zd$ is assigned a non-negative value $\tau_{\{x,y\}}$. The family $\{\tau_{\{x,y\}}:x,y\in\Zd,\;|x-y|=1\}$  is typically sampled from some translation-invariant probability distribution on functions from the nearest-neighbor edges of $\Zd$ to $[0,\infty)$; most often, the times are independent, following some common law $\mu$ supported on $[0,\infty)$. For any two vertices $x$ and $y$, the \textit{passage time} from $x$ to $y$ is then defined as
\begin{equation}\label{eq:random_metric} d(x,y)=\inf\left\{\sum_{k=1}^{|\pi|}\tau_{\{\pi_{k-1},\pi_k\}}:\pi\text{ is a path from $x$ to $y$}\right\}, \end{equation} 
where a path $\pi$ from $x$ to $y$ is a finite sequence $\pi=(\pi_0,\pi_1,\dots,\pi_n)$ with $n\in\mathbb N$, $\pi_0=x$, $\pi_n=y$, and $|\pi_k-\pi_{k-1}|=1$ for $k=1,\dots,n$, and  $|\pi| = n$ is its length. Thus defined, $d(\cdot,\cdot)$ is a random (pseudo-)metric on $\Zd$.

A \textit{growth process} can be defined from $d$ by letting
\begin{equation}\label{eq:growth_process}
\zeta_t(x) = \mathds{1}\{d(o,x) \leq t \},\quad t \geq 0,\; x \in \Zd.
\end{equation}
Very general shape theorems have been obtained for this process; see for instance \cite{GaretMarch12}. In the particular case in which the random variables $\tau_{\{x,y\}}$ are i.i.d. and exponentially distributed, it can be shown that $(\zeta_t)_{t \geq 0}$ is a Markov process on $\{0,1\}^\Zd$; it is called the \textit{Richardson model}, and can be seen as the contact process with no recoveries. That is, transitions from state 1 to state 0 are suppressed.

In this context of first passage percolation, a statement analogous to our Theorem \ref{mainThm} is that, if $|x|$ is large enough, then
\begin{equation}
\label{eq:analogous_statement}
\limsup_{n\to\infty} \P\left(d(o,nx) < d(o,(n+1)x) \right) > \frac12.
\end{equation}
This has been proved in \cite{haggstrom1998first} for the Richardson model on $\mathbb{Z}^2$ and for much more general passage-time distributions on $\Zd$ in \cite{GaretMarch05} and \cite{Hoffm05}.

Apart from the growth process \eqref{eq:growth_process}, the random metric in \eqref{eq:random_metric} can be used to define a \textit{competition process}. For this, fix two distinct \textit{sources} $x_1, x_2 \in \Zd$, let $$d(\{x_1,x_2\},x) = \min(d(x_1,x),d(x_2,x))$$ and define a process $\eta_t \in \{0,1,2\}^\Zd$ by setting
\begin{equation}
\label{eq:competition_process}
\eta_t(x) = \begin{cases}0&\text{if } t > d(\{x_1,x_2\},x),\\1&\text{if } t \leq d(\{x_1,x_2\},x),\; d(x_1,x) \le d(x_2,x),\\
2&\text{if }   t \leq d(\{x_1,x_2\},x),\; d(x_1,x) > d(x_2,x) \end{cases} \quad\qquad t \geq 0,\; x \in \Zd.
\end{equation}
This can be interpreted as follows: at time 0, a particle of type 1 is placed at $x_1$ and a particle of type 2 is placed at $x_2$; each type then invades adjacent sites as in the growth process \eqref{eq:growth_process}, with the rule that once a site has been taken by one of the types, it stays that way permanently. Then, $\Zd$ is partitioned into the sets
$$C_i = \{x\in\Zd:\;\lim_{t\to\infty}\eta_t(x) = i\},\quad i = 1,2.$$
The event $\{\#C_1=\#C_2 = \infty\}$ is called the \textit{coexistence event}, and one then wonders if it has positive probability. Letting $x_1= o$ and $x_2 = \bar{x} \in \Zd$, if the probability of coexistence with these two sources was zero, then symmetry would give
$$\P( \#C_1 = \infty,\;\#C_2 < \infty) = \P(\#C_1 < \infty,\;\#C_2 = \infty) = \frac12,$$
from which \eqref{eq:competition_process}  would allow us to conclude that
\begin{equation} \label{eq:first_impli}\lim_{n \to \infty} \P(d(o,n\bar{x}) < d(\bar{x},n\bar{x})) = \frac12,\end{equation}
which is equivalent to
\begin{equation} \label{eq:second_impli}\lim_{n \to \infty} \P(d(o,n\bar{x}) < d(o,(n+1)\bar{x})) = \frac12,\end{equation}
contradicting \eqref{eq:analogous_statement}, at least if $|\bar{x}|$ is large enough. Hence, for first-passage percolation, proving \eqref{eq:analogous_statement} is the key step in proving that coexistence can occur (and indeed, the aforementioned references \cite{haggstrom1998first, GaretMarch05, Hoffm05} were all primarily concerned with the coexistence problem).

A natural question, then, is whether our Theorem \ref{mainThm} can be reinterpreted as a coexistence result. The competition model that corresponds to the contact process in the same way as model \eqref{eq:competition_process} corresponds to the growth model \eqref{eq:growth_process} is Neuhauser's \textit{multitype contact process}, introduced in \cite{Neuha92}. However, the equivalence between \eqref{eq:first_impli} and \eqref{eq:second_impli}, which stems from the fact that both the growth process $(\zeta_t)_{t\geq 0}$ and the competition model $(\eta_t)_{t\geq 0}$ are defined from the same random metric, has no evident analogue for the contact process. Proving that the coexistence event has positive probability for the multitype contact process on $\Zd$ is still an open problem, apart from the easiest case of $d=1$ (see \cite{andjel2010survival} and \cite{valesin2010multitype}).

\subsubsection{Contact process in random environment.}
In a series of papers, the shape theorem has been extended to the \emph{contact process in random environment} by Garet and Marchand~\cite{GaretMarch12, GaretMarch14}. In this setting, the infection rate $\lambda$ is not constant. Rather, in \eqref{eqFlipRateDef}, we replace $\lambda$ with a random variable $\lambda_{\{x,y\}}$, and these random variables form an i.i.d.\ family with support in $[\lambda_{\mathsf{min}},\lambda_{\mathsf{max}}]$ where $0<\lambda_c<\lambda_{\mathsf{min}}<\lambda_{\mathsf{max}}<\infty$. Indeed, following the setup in \cite{GaretMarch12}, our results extend in a straightforward manner to this random setting. 

\bigskip

\subsection{Organization of the paper}
In the forthcoming section, we first prove Proposition \ref{propEssTightness}. This result gives strong quantitative estimates on the relation of $t(x)$ and $\sigma(x)$, and is the key to our proofs of Theorems \ref{mainThm} and \ref{thmTightness}, both of which we defer to Section \ref{thmSec}. 


\section{Proof of Proposition~\ref{propEssTightness}}

The challenge of Proposition~\ref{propEssTightness} is to establish bounds on the defects $\sigma(x) - \t(x)$ for sites $x$ far away from the origin.
The essence of the proof of Proposition~\ref{propEssTightness} is a carefully devised level-crossing argument leveraging the renewal structure of standard hitting times under the unconditioned measure $\P$. To achieve this goal, we introduce variants of the times $\{u_n, v_n\}_{n\ge1}$ from Section~\ref{essHitSec} that are tailor-made for the purpose of our proof.

More precisely, after fixing $L$ as in the statement of the proposition, we consider the ball 
$$B(x, R^2) = \{y \in \Zd:\, |y - x| \le R^2 \}$$
of radius $R^2$ around $x$, where $R$ is large and depends on $L$, but not on $x$. We pay special attention to the times at which the  infection hits each of the  $R$ shells 
$$\partial B(x,iR) = \{y \in \Zd: |y-x| = iR \},\quad 1\le i \le R. $$
If for some $i$ the space-time point where the infection hits the $i$th shell for the first time gives rise to a surviving contact process, then, by a renewal argument, it suffices to bound the essential hitting time of a point at distance at most $R^2$ from the origin. In particular, for this it is sufficient to invoke earlier location-dependent upper bounds.

The alternative is that upon hitting each of the $R$ shells, the infection started from the hitting point dies out. In the super-critical regime, a further renewal argument shows that this is an event of probability decaying geometrically in the number of shells.

We clearly need a due amount of care in order to make the renewal argument rigorous. In particular, it is important to know that the infections starting from the hitting points die out before the next shell is reached.

Due to the regenerative arguments sketched above, we need the location-dependent bounds from~\cite[Theorem 15]{GaretMarch14} not only if the infection is started from the origin, but from an arbitrary set $A \subseteq \Zd$ with $o \in A$. More precisely, we now put $u^A_0(x) = v^A_0(x) = 0$ and
\[u^A_{k+1}(x)=\inf\{t\ge v^A_k(x):x\in\xi^A_t\},\]
whereas the recursion for $v^A_{k+1}$ is unchanged, i.e., 
\[v^A_{k+1}(x)=\sup\{s\ge0:\; (x,u^A_{k+1}(x)) \rightsquigarrow \Zd \times \{s\} \}. \]
Finally, we set $\sigma^A(x)=u_{K^A(x)}(x) $ where 
\[K^A(x)=\min\{n\ge0:v^A_n(x)=\infty\text{ or }u^A_n(x)=\infty\}.\]


\begin{lemma}
	\label{cor20}
There exist $c$, $c' > 0$ such that for all $L > 0$, $x \in \Zd$ and $A \subseteq \Zd$ with $o \in A$,
\begin{align}\label{eqcor20}
\oP\left(\sigma^A(x) > L\right) \leq \exp\{-cL + c'|x|\}.
\end{align}
\end{lemma}
\begin{proof}
The result follows from Chebyshev's inequality and the claim that there exist $\beta, \gamma > 0$ such that, for all $x \in \Zd$ and all $A \subseteq \Zd$ with $o \in A$,
$$\oE(\exp\{\beta \sigma^A(x)\}) \leq \exp\{\gamma |x|\}. $$
	For the case $A = \{o\}$, this is \cite[Theorem 15]{GaretMarch14}. The statement with general $A \subseteq \Zd$ with $o \in A$ is proved in exactly the same way, so we do not go over the details.
\end{proof}

\begin{proof}[Proof of Proposition~\ref{propEssTightness}]
  Fix $L > 0$, which we may assume to be large throughout the proof without loss of generality. Define $R = \lfloor L^{1/4}\rfloor$. Let us first fix $x \in \Zd$ with $|x| \leq R^2$. Then, by Lemma~\ref{cor20},
\begin{equation}\oP(\sigma(x) - t(x) > L) \leq \oP(\sigma(x) > L) \stackrel{\eqref{eqcor20}}{\leq} \exp\left\{-cL + c'|x|\right\} < \exp\left\{-cL/2\right\}.\label{eq:zeroth_ing}\end{equation}

Now we fix $x$ with $|x| > R^2$ and aim to show that
\begin{equation}\label{eqDestBd}\P(\sigma(x) < \infty,\;\sigma(x) - t(x) > L) < C_0\exp\{-L^{\gamma_0}\}\end{equation}
for some $C_0,\gamma_0 > 0$; the desired bound then follows since $\oP(\sigma(x) < \infty) = 1$. As suggested in the outline of proof, we define the shell radius $r_i = R(R-i+1)$ and let
$$U_i = \inf\{t: \xi_t^o \cap B(x,r_i) \neq \varnothing\},\qquad i = 1,\ldots, R.$$
denote the first time that the infection hits the $i$th shell. On the event $\{U_i < \infty\}$, let $Z_i \in \partial B(x,r_i)$ denote the unique point of $\xi^o_{U_i} \cap B(x, r_i)$. Also on $\{U_i < \infty\}$, let
$$V_i = \inf\{s \geq U_i:\;(Z_i,U_i) \not\rightsquigarrow \Zd \times \{s\}\}$$
	denote the time at which  the infection starting from $(Z_i, U_i)$ dies out.

	Then there are three alternative scenarios under which the event $\{\sigma(x) < \infty,\;\sigma(x) - t(x) > L\}$ can occur: 
	\begin{enumerate}
		\item there exists at least one shell $\partial B(x, r_i)$ such that the infection generated by the first hitting point survives forever, i.e., $V_i = \infty$,
		\item for each $i \in \{1,\ldots, n\}$ the infection generated by the hitting point $(Z_i, U_i)$ dies out before the global infection reaches the $(i+1)$th shell, 
		\item one of these non-surviving infections lives longer than it takes the global infection to reach the next shell.
	\end{enumerate}
	More precisely,
\begin{equation}\label{eq:3terms}\begin{split}
\P(\sigma(x) < \infty,\; \sigma(x) - t(x) > L) &\leq \sum_{n=1}^R\P(U_n < \infty,\; V_n = \infty,\; \sigma(x) - t(x) > L) \\&+
 \P(U_1 < V_1 < U_2 < V_2 <\cdots < U_R < V_R < \infty)\\&
	+ \sum_{n=1}^{R-1} \P(U_{n+1} < V_{n} <\infty)
\end{split}\end{equation}
	We separately bound the three terms on the right-hand side. By repeated use of the Markov property, we have 
	\begin{equation}\P(U_1 < V_1 < U_2 < V_2 <\cdots < U_R < V_R < \infty)\leq(1-\rho)^R, \label{eq:first_ing}\end{equation} 
		where we recall that $\rho = \P_\lambda(o \rightsquigarrow \infty)$  denotes the survival probability. Next, for any $n$,
\begin{align}
&\P(U_n < \infty,\; V_n = \infty,\;\sigma(x) - t(x) > L) \leq \P(U_n < \infty,\;V_n = \infty,\;\sigma(x) - U_n > L) \nonumber\\
&=\sum_{y \in \partial B(x,r_n)}\; \sum_{\substack{A \subseteq \Zd,\\A\text{ finite}}} \P(U_n < \infty,\; Z_n = y,\; (y, U_n) \rightsquigarrow \infty,\; \xi^o_{U_n} = A,\; \sigma(x) - U_n > L)\nonumber\\
&= \sum_{y \in \partial B(x,r_n)}\; \sum_{\substack{A \subseteq \Zd,\\A \text{ finite}}} \P(U_n < \infty,\; Z_n = y,\;  \xi^o_{U_n} = A) \cdot \rho\cdot \oP(\sigma^{A-y}(x-y) > L)\nonumber\\
&\leq \rho \cdot \P(U_n < \infty) \cdot \exp\{-cL/2\},\label{eq:second_ing}
\end{align}
	the last inequality is obtained as in \eqref{eq:zeroth_ing} using $|x - y| \le R^2$. We now turn to the term inside the third sum on the right-hand side of \eqref{eq:3terms}. Since the contact process is in the super-critical regime, the survival time of a non-surviving infection started at a single site has exponential tails, cf.~\cite[Proposition 5]{GaretMarch12}, so that the Markov property yields
\begin{equation}\begin{split}\P(U_{n+1} < V_n < \infty) &\leq \P(V_n - U_n \in (\sqrt{R}, \infty)) + \P(U_{n+1} - U_n < 2\sqrt{R})\\
&\leq \exp\{-c\sqrt{R} \} + \P(\partial B(x,r_i) \times [0,2\sqrt{R}] \rightsquigarrow B(x,r_{i+1}) \times[0,2\sqrt{R}]).\end{split}\label{eq:third_ing}
\end{equation}
Since the propagation speed of the infection is at most linear outside an event of exponentially decaying probability ~\cite[Proposition 5]{GaretMarch12}, the second term on the right-hand side is less than
\begin{align}\nonumber&\#\partial B(x,r_i) \cdot 2d\lambda \cdot \int_0^{2\sqrt{R}} \max_{y \in \partial B(x,r_i)} \P\Big((y,s) \rightsquigarrow B(x,r_{i+1})\times [s,2\sqrt{R}] \Big)\; \mathrm{d}s
\\\nonumber&\leq \#\partial B(x,r_i) \cdot 2d\lambda \cdot 2\sqrt{R} \cdot \P\Big(\bigcup_{s \leq 2\sqrt{R}}\xi^y_s \nsubseteq B(y, R/2)\Big) \\&\leq \#\partial B(x,r_i) \cdot 2d\lambda \cdot 2\sqrt{R} \cdot \exp\{-c\sqrt{R}\}.\label{eq:third_ing0}
\end{align}
Combining the bounds \eqref{eq:3terms}, \eqref{eq:first_ing}, \eqref{eq:second_ing}, \eqref{eq:third_ing} and \eqref{eq:third_ing0} implies  \eqref{eqDestBd} and finishes the proof.
\end{proof}

	\section{Proof of Theorems~\ref{mainThm} and~\ref{thmTightness}}
\label{thmSec}

\subsection{Proof of Theorem~\ref{mainThm}}

Before getting into the details, we provide a rough outline of the proof. To prove Theorem~\ref{mainThm}, we consider $\oE[\t(kx) - \t((k-1)x)]$ as discrete derivatives of the expected hitting time 
$\oE\t(n x)$, in the sense that 
$$\oE\t(n x)  = \sum_{k = 1}^n \oE[\t(k x) - \t((k-1)  x)].$$
The shape theorem for the supercritical contact process shows that the left-hand side is of the order $n\mu(x)$. Hence, a large number of small derivatives $\oE[\t((n+1) x) - \t(n x)]$ would need to be compensated by a substantial number of atypically large derivatives. However, leveraging the concept of essential hitting times allows us to prove the following sub-additivity bound that is sufficiently strong to exclude excessive growth.

\begin{lemma}
	\label{derivBoundLem}
	Let $p > 0$ and $x, y \in \Zd$ be arbitrary. Then, 
	$$\oE\left[|t(x) - \sigma(y)|^p\cdot \mathds{1}\{\t(x) \ge \s(y) \} \right] \leq \oP(\t(x) \ge \s(y)) \cdot \oE \left[t(x-y)^p \right] $$
\end{lemma}

Before proving Lemma~\ref{derivBoundLem}, we explain how it implies Theorem~\ref{mainThm}. For this, we need an $L_p$-version of the shape theorem~\cite[Theorem 3]{GaretMarch12}.
\begin{lemma}
	\label{shapeThm}
	Let $p > 0$ be arbitrary. Then, $$\lim_{|x| \to \infty}\frac{\oE\left[t(x)^p \right]}{\mu(x)^p} = 1.$$
\end{lemma}
\begin{proof}
By Proposition~\ref{propEssTightness} it suffices to prove the claim with $\t$ replaced by $\s$. The shape theorem~\cite[Theorem 3]{GaretMarch12} gives almost sure convergence of $\sigma(x)/\mu(x)$ as $|x| \to \infty$. By~\cite[Theorem 15]{GaretMarch14}, $\{\sigma(x)/\mu(x)\}_{x \in \Zd \setminus \{o\}}$ is bounded in $L_p$ for every $p$; hence, for every $p$, $\{\sigma(x)^p/\mu(x)^p\}_{x \in \Zd \setminus \{o\}}$ is uniformly integrable. This concludes the proof.
\end{proof}

\begin{proof}[Proof of Theorem~\ref{mainThm}]
	First, for any $n \ge 1$ and $x\in\Zd$, 
	$$\oE[\t(nx)] = \sum_{k = 1}^n \oE[\t(kx) - \t((k-1)x)] \le \sum_{k = 1}^n \oE[\t(kx) - \t((k-1)x)] \one \{\oE[\t(kx)] \ge \oE[\t((k-1)x)]\}.$$
	Hence, introducing 
	$$I_{n,x} = \{k \in \{1,\ldots, n \}:\, \oE[\t(kx)] \ge \oE[\t((k-1)x)]\}$$
	as the index set of correctly ordered expected hitting times, we arrive at 
		$$\oE[\t(nx)] \le \sum_{k \in I_{n,x}} \oE[\t(kx) - \s((k-1)x)] + \sum_{k \in I_{n,x}}\oE[\s((k-1)x) - \t((k-1)x)]$$
	By Lemma~\ref{derivBoundLem}, the first sum is bounded above by 
		$$\oE[\t(x)]\sum_{k \in I_{n,x}}\oP(\s((k-1) x) \le \t(k x)) .$$
		Additionally, by Proposition~\ref{propEssTightness}, the second sum is bounded from above by a universal constant $C_1$. Hence, by Lemma~\ref{shapeThm}, 
		\begin{align*}
			\mu(x) &= \lim_{n \to \infty}n^{-1} \oE \t(nx) \\ 
			&\le C_1 + \oE[\t(x)] \liminf_{n \to \infty}\frac{\sum_{k \in I_{n,x}}\oP(\s((k-1) x) \le \t(k x)) }{n}\\
			&\le C_1 + \oE[\t(x)] \liminf_{n \to \infty}\frac{\sum_{k \in I_{n,x}}\oP(\t((k-1) x) \le \t(k x)) }{n}.
		\end{align*}
Now, dividing by $\mu(x)$ and taking the limit $|x| \to \infty$, the two statements of the theorem respectively follow from the trivial bounds
\begin{align*}& \sum_{k \in I_{n,x}}\oP(\t((k-1) x) \le \t(k x)) \leq \sum_{k=1}^n\oP(\t((k-1) x) \le \t(k x)), \\& \sum_{k \in I_{n,x}}\oP(\t((k-1) x) \le \t(k x)) \leq \#I_{n,x}.
\end{align*}
\end{proof}

It remains to prove Lemma~\ref{derivBoundLem}.

\begin{proof}[Proof of Lemma~\ref{derivBoundLem}]
	Let 
	$$\t(y, x) = \inf\{s \geq \s(y):\;(y, \s(y)) \rightsquigarrow (x, s)\}$$
	denote the first time that the infection from the space-time point $(y, \s(y))$ reaches $x$. By definition of $\s(y)$, this hitting time is almost surely finite under $\oP$. Moreover, $\t(y, x)$ is measurable with respect to the Harris construction \emph{after} time $\s(y)$, whereas $\{\t(x) \ge \s(y)\}$ is measurable with respect to the Harris construction \emph{before} time $\s(y)$. Hence, the renewal-type property of $\s(y)$ as in \cite[Lemma 8]{GaretMarch12} gives  
	\begin{align*}
	\oE\left[|t(x) - \sigma(y)|^p\cdot \mathds{1}\{t(x) \geq \sigma(y)\}\right] &\le \oE\left[t(y,x)^p \cdot \mathds{1}\{t(x) \geq \sigma(y)\} \right]\\& = \oP\left(t(x) \geq \sigma(y) \right) \cdot \oE\left[t(x-y)^p\right],
	\end{align*}
which proves the claim.

\end{proof}

\subsection{Proof of Theorem~\ref{thmTightness}}

In order to prove Theorem~\ref{thmTightness}, we note that $\t(x) - \t(y)$ can be large only if either $\t(x) - \s(y)$ or $\s(y) - \t(y)$ is large. In particular, with Lemma~\ref{derivBoundLem} and Proposition~\ref{propEssTightness} at our disposal, a quick algebraic manipulation gives us Theorem~\ref{thmTightness}.

\begin{proof}[Proof of Theorem~\ref{thmTightness}]
By the inequality
$$\oE\left[\left|\frac{t(x)-t(y)}{|x-y|} \right|^p \right] \leq 1 +  \oE\left[\left|\frac{t(x)-t(y)}{|x-y|} \right|^{p'} \right]$$
for $0 < p < p'$, it suffices to prove the statement for $p \geq 1$.	
	
	By Minkowski inequality,
\begin{align*}
&\left(\oE\left[|t(x) - t(y)|^p\right]\right)^{1/p}\\&\quad  \leq \left(\oE \left[(t(x) - t(y))^p \cdot \mathds{1}\{t(x) \ge t(y)\}\right]\right)^{1/p} + \left(\oE\left[(t(y) - t(x))^p \cdot \mathds{1}\{t(y) \ge t(x) \} \right]\right)^{1/p},
\end{align*}
	so that we have reduced the task to establishing an upper bound for the first summand. By distinguishing further between the events $\{\t(x) \ge \s(y)\}$ and $\{\t(y) \le \t(x) \le \s(y)\}$, and applying Lemma~\ref{derivBoundLem}, we arrive at 
	\begin{align*}
	&\left(\oE\left[|t(x) - t(y)|^p\cdot \mathds{1}\{t(x) \ge t(y) \}\right]\right)^{1/p}\\
	&\leq \left(\oE\left[|t(x) - \sigma(y)|^p\cdot \mathds{1}\{t(x) \ge t(y) \}\right]\right)^{1/p} + \left(\oE\left[| \sigma(y)-t(y)|^p\right]\right)^{1/p}\\
		&\leq \left(\oE\left[|t(x) - \sigma(y)|^p\cdot \mathds{1}\{t(x) \ge \sigma(y) \}\right]\right)^{1/p} + 2\left(\oE\left[| \sigma(y)-t(y)|^p\right]\right)^{1/p}\\
	& \leq \left(\oE\left[t(x-y)^p \right] \right)^{1/p} + 2\left(\oE\left[(\sigma(y) - t(y))^p \right] \right)^{1/p}.
	\end{align*}
	An application of Lemma~\ref{shapeThm} and Proposition~\ref{propEssTightness} then concludes the proof.
\end{proof}

%


\bibliography{refs}

\begin{thebibliography}{10}

\bibitem{andjel2010survival}
E.~D. Andjel, J.~R. Miller, and E.~Pardoux.
\newblock Survival of a single mutant in one dimension.
\newblock {\em Electron. J. Prob.}, 15:386--408, 2010.

\bibitem{BezuiGrimm90}
C.~Bezuidenhout and G.~Grimmett.
\newblock The critical contact process dies out.
\newblock {\em Ann. Probab.}, 18(4):1462--1482, 1990.

\bibitem{BirknCernyDeppeGante13}
M.~Birkner, J.~{\v C}ern\'y, A.~Depperschmidt, and N.~Gantert.
\newblock Directed random walk on the backbone of an oriented percolation
  cluster.
\newblock {\em Electron. J. Probab.}, 18:no. 80, 35, 2013.

\bibitem{Durre88}
R.~Durrett.
\newblock {\em Lecture Notes on Particle Systems and Percolation}.
\newblock Wadsworth \& Brooks, Pacific Grove, CA, 1988.

\bibitem{Durre91}
R.~Durrett.
\newblock The contact process, 1974--1989.
\newblock In {\em Mathematics of Random Media}, pages 1--18. Amer. Math. Soc.,
  Providence, RI, 1991.

\bibitem{DurreGriff82}
R.~Durrett and D.~Griffeath.
\newblock Contact processes in several dimensions.
\newblock {\em Z. Wahrsch. Verw. Gebiete}, 59(4):535--552, 1982.

\bibitem{GaretMarch05}
O.~Garet and R.~Marchand.
\newblock Coexistence in two-type first-passage percolation models.
\newblock {\em Ann. Appl. Probab.}, 15(1A):298--330, 2005.

\bibitem{GaretMarch12}
O.~Garet and R.~Marchand.
\newblock Asymptotic shape for the contact process in random environment.
\newblock {\em Ann. Appl. Probab.}, 22(4):1362--1410, 2012.

\bibitem{GaretMarch14}
O.~Garet and R.~Marchand.
\newblock Large deviations for the contact process in random environment.
\newblock {\em Ann. Probab.}, 42(4):1438--1479, 2014.

\bibitem{GaretMarchThere15}
O.~Garet, R.~Marchand, and M.~Th\'eret.
\newblock Continuity of the asymptotic shape of the supercritical contact
  process.
\newblock {\em Electron. Commun. Probab.}, 20:no. 92, 11, 2015.

\bibitem{haggstrom1998first}
O.~H{\"a}ggstr{\"o}m and R.~Pemantle.
\newblock First passage percolation and a model for competing spatial growth.
\newblock {\em J. Appl. Probab.}, 35(03):683--692, 1998.

\bibitem{Hoffm05}
C.~Hoffman.
\newblock Coexistence for {R}ichardson type competing spatial growth models.
\newblock {\em Ann. Appl. Probab.}, 15(1B):739--747, 2005.

\bibitem{Ligge85}
T.~M. Liggett.
\newblock {\em Interacting Particle Systems}.
\newblock Springer-Verlag, New York, 1985.

\bibitem{Ligge99}
T.~M. Liggett.
\newblock {\em Stochastic Interacting Systems: Contact, Voter and Exclusion
  Processes}.
\newblock Springer-Verlag, Berlin, 1999.

\bibitem{Neuha92}
C.~Neuhauser.
\newblock Ergodic theorems for the multitype contact process.
\newblock {\em Probab. Theory Related Fields}, 91(3-4):467--506, 1992.

\bibitem{valesin2010multitype}
D.~Valesin.
\newblock Multitype contact process on {Z}: extinction and interface.
\newblock {\em Electron. J. Probab.}, 15:2220--2260, 2010.

\end{thebibliography}
\bibliographystyle{abbrv}

\end{document}